\documentclass[a4paper,twoside,12pt]{article}
\usepackage{amssymb,amsfonts,amsmath,amsthm,latexsym}
\usepackage[pdftex,bookmarks,colorlinks=false]{hyperref}
\usepackage[hmargin=1.25in,vmargin=1.25in]{geometry}
\usepackage[auth-sc-lg,affil-sl]{authblk}
\usepackage{graphicx,caption,subcaption}
\newtheorem{thm}{Theorem}[section]
\newtheorem{cor}[thm]{Corollary}
\newtheorem{defn}[thm]{Definition}

\newtheorem{prob}[thm]{Problem}
\newtheorem{prop}[thm]{Proposition}

\newtheorem*{ill}{Illustration}

\numberwithin{equation}{section}

\def\ni{\noindent}
\def\N{\mathbb{N}}
\def\M{\mathbb{M}}

\def\cS{\mathcal{S}}

\title{\textbf{\sc  A Study on Set-Graphs}}

\author{Johan Kok}
\affil{\small Tshwane Metropolitan Police Department\\ City of Tshwane, Republic of South Africa\\ E-mail: kokkiek2@tshwane.gov.za}

\author{N. K. Sudev}
\affil{\small Department of Mathematics\\ Vidya Academy of Science \& Technology \\ Thalakkottukara, Thrissur - 680501, India.\\ E-mail: sudevnk@gmail.com}

\author{K. P. Chithra}
\affil{\small Naduvath Mana, Nandikkara \\ Thrissur - 680301, India.\\ E-mail: chithrasudev@gmail.com}

\author{C. Susanth}
\affil{\small Department of Mathematics\\ Vidya Academy of Science \& Technology \\ Thalakkottukara, Thrissur - 680501, India.\\ E-mail: susanth\_c@yahoo.com}

\date{}

\begin{document}
\maketitle

\begin{abstract}
A \textit{primitive hole} of a graph $G$ is a cycle of length $3$ in $G$. The number of primitive holes in a given graph $G$ is called the primitive hole number of that graph $G$. The primitive degree of a vertex $v$ of a given graph $G$ is the number of primitive holes incident on the vertex $v$. In this paper, we introduce the notion of set-graphs and study the properties and characteristics of set-graphs.  We also check the primitive hole number and primitive degree of set-graphs. Interesting introductory results on the nature of order of set-graphs, degree of the vertices corresponding to subsets of equal cardinality, the number of largest complete subgraphs in a set-graph etc. are discussed in this study. A recursive formula to determine the primitive hole number of a set-graph is also derived in this paper.
\end{abstract}

\ni \textbf{Key Words:} set-graphs, primitive hole, primitive degree.
 
\vspace{0.2cm}

\noindent \textbf{MS Classification}: 05C07, 05C38, 05C78.

\section{Introduction}

For general notations and concepts in graph theory, we refer to \cite{BM}, \cite{FH} and \cite{DBW}. All graphs mentioned in this paper are simple, connected undirected and finite, unless mentioned otherwise.

\vspace{0.25cm}

A \textit{hole} of a simple connected graph $G$ is a chordless cycle $C_n$ , where $n \in  N$, in $G$. The \textit{girth} of a simple connected graph $G$, denoted by $g(G)$, is the order of the smallest cycle in $G$. The following notions are introduced in \cite{KS1}.

\begin{defn}{\rm
		\cite{KS1} A \textit{primitive hole} of a graph $G$ is a cycle of length $3$ in $G$. The number of primitive holes in a given graph $G$ is called the \textit{primitive hole number} of that graph $G$. The primitive hole number of a graph $G$ is denoted by $h(G)$. }
\end{defn}

\begin{defn}{\rm
		\cite{KS1} The \textit{primitive degree} of a vertex $v$ of a given graph $G$ is the number of primitive holes incident on the vertex $v$ and the primitive degree of the vertex $v$ in the graph $G$ is denoted by $d^p_G(v)$.}
\end{defn}

Some studies on primitive holes of certain graphs have been made in \cite{KS1}. The number of primitive holes in certain standard graph classes, their line graphs and total graphs were determined in this study. Some of the major results proved in \cite{KS1} are the following.  

\begin{thm}
	{\rm \cite{KS1}} The number of primitive holes in a complete graph $K_n$ is $h(K_n)=\binom{n}{3}$.
\end{thm}

\begin{thm}
	{\rm \cite{KS1}} For any subgraph $H$ of a graph $G$, we have $h(H)\le h(G)$. Moreover, if $G$ is a graph on $n$ vertices, then $0\le h(G)\le \binom{n}{3}$. 
\end{thm}

\section{Set-Graphs}

In this paper, we introduce the notion of set-graphs and study certain characteristics of set-graphs and also present a number of interesting results related to graph properties and invariants. A set-graph is defined as follows.

\begin{defn}\label{D-SG}{\rm 
		Let $A^{(n)} = \{a_1, a_2, a_3, \ldots, a_n\}, n\in \N$ be a non-empty set and the $i$-th $s$-element subset of $A^{(n)}$ be denoted by $A_{s,i}^{(n)}$.  Now consider $\cS = \{ A_{s,i}^{(n)}: A_{s,i}^{(n)} \subseteq A^{(n)}, A_{s,i}^{(n)} \ne \emptyset\}$. The \textit{set-graph} corresponding to set $A^{(n)}$, denoted $G_{A^{(n)}}$, is defined to be the graph with $V(G_{A^{(n)}}) = \{v_{s,i}: A_{s,i}^{(n)} \in \cS\}$ and $E(G_{A^{(n)}}) = \{v_{s,i}v_{t,j}:~ A_{s,i}^{(n)}\cap A_{t,j}^{(n)}\ne \emptyset\}$, where $s\ne t~ \text{or}~ i\ne j$. }
\end{defn} 

It can be noted from the definition of set-graphs that $A^{(n)}\ne\emptyset$ and if $|A^{(n)}|$ is a singleton, then $G_{A^{(n)}}$ to be the trivial graph. Hence, all sets we consider here are non-empty, non-singleton sets.

\vspace{0.25cm}

Let us now write the vertex set of a set-graph $G_{A^{(n)}}$ as $V(G_{A^{(n)}})=\{v_{s,r}: 1\le r\le \binom{n}{s}\}$, where $s$ is the cardinality of the subset $A^{(n)}_{s,r}$ of $A^{(n)}$ corresponding to the vertex $v_{s,r}$. 

\vspace{0.25cm}

\ni The following result is perhaps obvious, but an important property of set-graphs.

\begin{prop}
	If $G$ is a set-graph, then $G$ has odd number of vertices.
\end{prop}
\begin{proof}
	Let $G$ be a set-graph with respect to the set $A^{(n)}$. It is to be noted the number of non-empty subsets of $A^{(n)}$ is $2^n-1$. Since every vertex of $G$ corresponds to a non-empty subset of $A^{(n)}$, the number of vertices in $G$ must be $2^n-1$, an odd integer. 
\end{proof}

\begin{ill}{\rm 
		Consider the set-graph with respect to the set $A^{(3)} = \{a_1, a_2, a_3\}$. Here we have the subsets of $A^{(3)}$ which are $A^{(3)}_{1,1}=\{a_1\}, A^{(3)}_{1,2}=\{a_2\}, A^{(3)}_{1,3}=\{a_3\}, A^{(3)}_{2,1}=\{a_1, a_2\}, A^{(3)}_{2,2}=\{a_1, a_3\},A^{(3)}_{2,3} =\{a_2, a_3\}, A^{(3)}_{3,1}=\{a_1, a_2, a_3\}$. Then, the vertices of $G_{A^{(3)}}$ have the labeling as follows. $v_{1,1} =\{a_1\}, v_{1,2} =\{a_2\}, v_{1,3} =\{a_3\}, v_{2,1}=\{a_1, a_2\}, v_{2,2} = \{a_1, a_3\}, v_{2,3} =\{a_2, a_3\}, v_{3,1} = \{a_1, a_2, a_3\}$.}
\end{ill}

\ni Figure \ref{fig-1} depicts the above mentioned labeling procedure of the set-graph $G_{A^{(3)}}$. 

\begin{figure}[h!]
	\centering
	\includegraphics[width=0.5\linewidth]{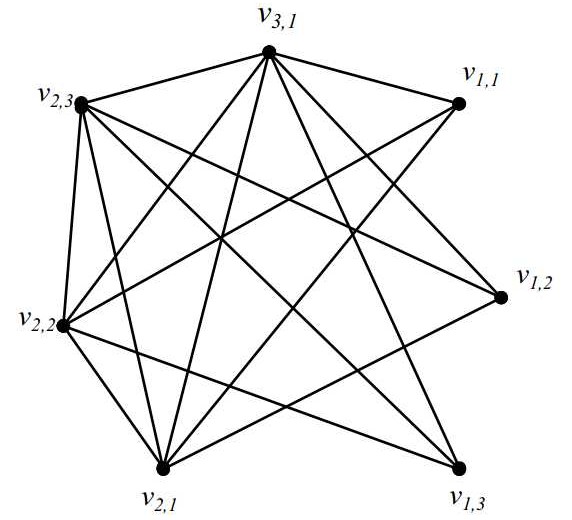}
	\caption{}
	\label{fig-1}
\end{figure}

\begin{thm}\label{T-SGDV}
	Let $G_{A^{(n)}}$ be a set-graph. Then, the vertices $v_{s,i}, v_{s,j}$ of $G_{A^{(n)}}$, corresponding to subsets $A^{(n)}_{s,i}$ and $A^{(n)}_{s,j}$ in $\cS$ of equal cardinality, have the same degree in $G_{A^{(n)}}$. That is, $d_{G_{A^{(n)}}}(v_{s,i}) = d_{G_{A^{(n)}}}(v_{s,j})$.
\end{thm}
\begin{proof}
	Consider the set-graph $G=G_{A^{(n)}}, n \in \N$. We begin by considering the vertices of $G$ corresponding to the $n$ singleton subsets of $A^{(n)}$. Let these vertices be denoted by $v_{1,i}$, where $1\le i\le n$. Clearly, for all ${j\ne i}$, we have $\{a_i\} \cap \{a_k\}=\emptyset$. Hence, by the definition of set-graphs, it follows that no edges are induced amongst the vertices $v_{1,1}, v_{1,2}, v_{1,3},\ldots, v_{1,n}$.
	
	\vspace{0.2cm}
	
	Now, construct all the two element subsets of $A^{(n)}$. Now choose two arbitrary vertices $v_{2,i}$ and $v_{2,j}$, where $i\ne j$.  Then, here we have the subsets of $A^{(n)}$ of the form $\{a_i, a_j\}$, for  $1\le i\ne j \le n$. It can be observed that the subsets of the form $\{a_i, a_j\}$ and $\{a_j, a_k\}$ are the elements of $\cS$, where $1 \le i\ne j \ne k \le n$. Moreover, $\{a_i\} \cap \{a_i, a_j\}\ne \emptyset$ for all $1 \le i \ne j \le n$. In a similar way, we can extend this argument for the sets $\{a_i\}$ and an arbitrary subset of $A^{(n)}$ containing the element $a_i$. That is, the vertex $v_{1,i}$ is adjacent to those vertices of $G$ whose corresponding sets have $m$ elements including the common elements $a_i$, for $m\ge 2$. Therefore, $d_G(v_{1,i})=2^{n-1}-1$.  Since the choice of $i$ is arbitrary, we have $d_G(v_{1,i})=d_G(v_{1,j})=2^{n-1}-1$ for all $1\le i,j \le n$. Therefore, the result holds for $s=1$. 
	
	\vspace{0.2cm}
	
	Now, assume that the result holds for $s=k$, where $k$ is a positive integer. That is, we have $d_G(v_{k,i})=d_G(v_{k,j})$ for all $1\le i,j \le \binom{n}{k}$. 
	
	\vspace{0.2cm}
	
	Next, consider the vertices of $G$ corresponding to the $(k+1)$-element subsets of $A^{(n)}$. Let $A_{(k+1),i}^{(n)}$ be a $(k+1)$-element subset of $A^{(n)}$ and let $v_{(k+1),i}$ be the vertex of $G$ corresponding to the set $A_{(k+1),i}^{(n)}$. Let $a_l$ be an arbitrary element of the set $A_{(k+1),i}^{(n)}$ and let $A_{(k+1),i}^{(n)'}=A_{(k+1),i}^{(n)}-\{a_l\}$. Then, the vertex $v_{(k+1),i}$ is adjacent to the vertices of $G$ corresponding to the sets containing the element $a_l$ in addition to the vertices of $G$ corresponding to the proper subsets of $A_{(k+1),i}^{(n)}$ and $A_{(k+1),i}^{(n)'}$. Hence, the difference between the number of edges incident on $v_{(k+1),i}$ and the number of edges incident on the vertex $v'_{(k+1),i}$ corresponding to the set $A_{(k+1),i}^{(n)'}$ is equal to the number of subsets of $A^{(n)}$ containing the element $a_l$, other than $A_{(k+1),i}^{(n)}$. This number is a constant for any set of $(k+1)$-element sets. Therefore, $d_G(v_{(k+1),i})=d_G(v_{(k+1),j})$ for all $1\le i,j \le \binom{n}{k+1}$. That is, the result is true for $s=k+1$ if it is true for $s=k$. Therefore, the theorem follows by induction.
\end{proof}

A question that arouses much interest in this context is what the degree of an arbitrary vertex  of a set-graph $G_{A^{(n)}}$. The following result provides a solution to this problem.

\begin{thm}
	Let $G$ be a set-graph with respect to a non-empty set $A^{(n)}=\{a_1,a_2,a_3,\ldots,a_n\}$ and let $v_{k,i}$ be an arbitrary vertex of $G$ corresponding to an $k$-element subset of $A^{(n)}$. Then, $d_G(v_{k,i})= (\sum\limits_{J}(-1)^{|J|-1}|\bigcap\limits_{j\in J}\cS_j|)-1$, where $J$ is an indexing set such that $\emptyset\ne J\subseteq \{0,1,2,\ldots,k\}$ and $\cS_j$ is the collection of subsets of $A^{(n)}$ containing the element $a_j$.
\end{thm}
\begin{proof}
	Let $G$ be a set-graph with respect to a non-empty set $A^{(n)}$. Without loss of generality, let $A_{k,i}^{(n)}$ be a $k$-element subset of $A^{(n)}$, say $\{a_1,a_2,a_3,\ldots,a_k\}$ and let $v_{k,i}$ be the vertex of $G$ corresponding to the set $A_{k,i}^{(n)}$. Therefore, the vertex $v_{k,i}$ is adjacent to the vertices of $G$ which correspond to the subsets of $A^{(n)}$, containing the at least one element of $A_{k,i}^{(n)}$. That is $d_G(v_{k,i})= |\bigcup\limits_{j\in J}\cS_j|-1$. But, by principle of inclusion and exclusion of sets, we have $|\bigcup\limits_{j\in J}\cS_j|=\sum\limits_{J}(-1)^{|J|-1}|\bigcap\limits_{j\in J}\cS_j|$, where $\emptyset\ne J\subseteq \{0,1,2,\ldots,k\}$. Therefore, $d_G(v_{k,i})= (\sum\limits_{J}(-1)^{|J|-1}|\bigcap\limits_{j\in J}\cS_j|)-1$, where $\emptyset\ne J\subseteq \{0,1,2,\ldots,k\}$.
\end{proof}

Determining the degree of vertices of a set-graph is an important and interesting problem at this time. The following result determines a lower and upper limits for the degree of vertices of a given set-graph.

\begin{thm}\label{T-SGDV1}
	For any vertex $v_{s,i}$ of a set-graph $G=G_{A^{(n)}}$, we have  $2^{n-1}-1\le d_G(v_{s,i})\le 2(2^{n-1}-1)$. 
\end{thm}
\begin{proof} Let $G=G_{A^{(n)}}$ be a set-graph with respect to a non-empty set $A^{(n)}$.  Here, we need to consider the following two cases.
	
	\vspace{0.2cm}
	
	\ni {\em Case-1:} It is to be noted that the vertices of $G$ corresponding to singleton subsets of $G$ have the minimum degree in $G$. Without loss of generality, let the vertex $v_{s,i}$ of $G$ corresponds to the set $\{a_i\}$. Then, $v_{s,i}$ should be adjacent to the vertices of $G$ corresponding to the subsets of $A^{(n)}$, other than itself, containing the element $a_i$. Therefore, degree of the vertex $v_{s,i}$ is equal to the the number of $m$-element subsets of $A^{(n)}$ containing the element $a_i$ for $m\ge2$. By binomial theorem, the total number of subsets of an $n$-element set, containing a particular element is $2^{n-1}$. Therefore, the minimum degree of a vertex in $G$ is $2^{n-1}-1$.
	
	\vspace{0.2cm}
	
	\ni {\em Case-2:} Note that we need to consider $2^n -1$ of the subsets of $A^{(n)}$ only excluding $\emptyset$. Hence, the final vertex  $v_{n,1}$ of the graph $G=G_{A^{(n)}}$ corresponding to the set $A^{(n)}$ in $\cS$ will be adjacent to all its preceding vertices. Since $G$ has $2^n-1$ vertices, $d_G(v_{n,1})=2n-2$. No other vertices in $G$  can be adjacent to all other vertices of $G$, the vertex $v_{n,1}$ has the maximum possible degree in $G$. That is, the maximum degree of a vertex in $G$ is $2^n-2=2(2^{n-1}-1)$.  
\end{proof}

\ni The following results are immediate consequences of the above theorem.

\begin{cor}
	For any set-graph $G=G_{A^{(n)}}$, $\Delta(G)=2\,\delta(G)$. 
\end{cor}
\begin{proof}
	From the proof the above theorem, we have $\delta (G)= 2^{n-1}-1$ and $\Delta (G) = 2^n-2= 2(2^{n-1}-1)$. This completes the proof.
\end{proof}

\begin{cor}\label{C-UDG}
	There exists a unique vertex $v_{n,1}$ in a set-graph $G_{A^{(n)}}$ having degree $\Delta(G_{A^{(n)}})$.
\end{cor}
\begin{proof}
	The proof follows from Case-2 of Theorem \ref{T-SGDV1}.
\end{proof}

The following result indicates the nature of the minimal and maximal degrees of the vertices of a set-graph.

\begin{cor}
	The maximal degree of vertex in a set-graph $G$ is always an even number and the minimal degree of a vertex in $G$ is always an odd number.
\end{cor}
\begin{proof}
	Let $G$ be a set-graph with respect to a non-empty set $A^{(n)}$. Then, by Theorem \ref{T-SGDV1}, the maximum degree of a vertex in a set-graph $G$ is $\Delta (G)=2(2^{n-1}-1)$, which is always an even number and the minimal degree of a vertex $G$ is $\delta(G)=2^{n-1}-1$, which is always an odd number.
\end{proof}

We have already proved that the vertex of the set-graph $G$ corresponding to the set $A^{(n)}$ itself has the maximum degree $2^n-2$ in $G$. Analogous to this result, we propose the following result on the primitive degree of this vertex $v_{n,1}$.

\begin{prop}
	For set-graph $G=G_{A^{(n)}}$, the primitive degree of the vertex corresponding the set $A^{(n)}$ is $d^p_{G}(v_{n,1}) = |E(G)|-\Delta(G)$.
\end{prop}
\begin{proof}
	Let $G=G_{A^{(n)}}$ be a set-graph with respect to the set $A^{(n)}$. Consider the subgraph $G'= G-v_{n,1}$. By Theorem \ref{T-SGDV}, we have $d(v_{n,1})=2^n-2$ and hence $|E(G')| = |E(G)|-(2^n - 2)$. Furthermore, both ends ends of every edge $vu \in E(G')$ are adjacent to vertex $v_{n,1}$ in $G_{A^{(n)}}$. Hence, each such edge $uv$ in $G'$ corresponds to a primitive hole $C_3$ in $G_{A^{(n)}}$ on the vertices $u,v, v_{n,1}$. Hence, $d^p_G(v_{n,1})=|E(G)|-(2^n - 2)= |E(G)|-\Delta(G)$. 
\end{proof}

In the result given below, we describe a recursive formula to determine the number of edges of a set-graph. 

\begin{thm}\label{T-SGRD}
	For a set-graph $G_{A^{(n+1)}}$ we have 
	\begin{enumerate}\itemsep0mm
		\item[(i)] $|E(G_{A^{(n+1)}})| = 3|E(G_{A^{(n)}})| + |V(G_{A^{(n)}})| + |E(K_{|V(G_{A^{(n)}})| + 1})|$
		\item[(ii)] $|V(G_{A^{(n+1)}})|  = 2|V(G_{A^{(n)}})| +1$.
	\end{enumerate}
\end{thm}
\begin{proof}
	Consider the set-graph $G_{A^{(n)}}$. To extend it to $G_{A^{(n+1)}}$, we proceed in five steps as explained below.
	\begin{enumerate}\itemsep0.25cm
		\item[(i)] Replicate the vertices of $G_{A^{(n)}}$ as an \textit{edgeless} graph and add the new element $a_{n+1}$ as an element to all subsets corresponding to all $v_{s,i} \in V(G_{A^{(n)}})$ and label each \textit{replica vertex}, $v^{\ast}_{s,i}$. Also add the new vertex $v_{1,(n+1)}$ corresponding to the single element subset $\{a_{n+1}\}$.,
		
		\item[(ii)] Apply the definition of a set-graph to these new vertices. Clearly, we obtain the complete graph $K_{|V(G_{A^{(n)}})| + 1}$.
		
		\item[(iii)] Each vertex $v_{s,i} \in V(G_{A^{(n)}})$ corresponding to the subset $A_{s,i}^{(n)}$ can be linked with its replica vertex corresponding to the new subset $A_{s,i}^{(n)} \cup \{a_{n+1}\}$. We refer to \textit{parallel linkage} and exactly $|V(G_{A^{(n)}})|$ such edges are added.
		
		\item[(iv)] For the ends of each edge in $E(G_{A^{(n)}})$ say $v_{s,l}$ and $v_{t,l'}$ with corresponding subsets say, $A_{s,k}^{(n)}$ and $A_{t,m}^{(n)}$ we have $A_{s,k}^{(n)} \cap (A_{t,m}^{(n)}\cup \{a_{n+1}\}) \ne \emptyset$ and $A_{t,m}^{(n)}\cap (A_{s,k}^{(n)} \cup \{a_{n+1}\}) \ne \emptyset$. So the edges $v_{s,l} v_{s,l}^{\ast}$ and $v_{t,l'}v_{t,l'}^{\ast}$, with $v_{s,l}^{\ast}, v_{t,l'}^{\ast}$ corresponding to subsets $A_{s,k}^{(n)} \cup \{a_{n+1}\}$ and $A_{t,m}^{(n)} \cup \{a_{n+1}\}$ respectively, exist. Hence $2|E(G_{A^{(n)}})|$ additional edges are linked.
		
		\item[(v)] Relabel the vertices according to the Definition \ref{D-SG} to obtain the set-graph $G_{A^{(n+1)}}$. 
	\end{enumerate} 
	
	The summation of the edges added through steps (i) to (v) plus the existing edges of $G_{A^{(n)}}$ provides the result:  $|E(G_{A^{(n+1)}})| = 3|E(G_{A^{(n)}})| + |V(G_{A^{(n)}})| + |E(K_{|V(G_{A^{(n)}})| + 1})|$.  
	
	\vspace{0.2cm}
	
	The second part of the Theorem is an immediate consequence of the above proof of first part. Then, the proof is complete.
\end{proof}

\ni The following is a result related to the largest complete graphs found in a set-graph.

\begin{prop}\label{P-SGKN}
	The set-graph $G_{A^{(n)}}, n \ge 2$ has exactly two largest complete graphs, $K_{2^{n-1}}$. 
\end{prop}
\begin{proof}
	Consider the set-graph $G_{A^{(n-1)}}$ and extend to $G_{A^{(n)}}$. From step (ii) in the proof of Theorem \ref{T-SGRD}, we construct a largest complete graph amongst the replica vertices and the new vertex $v_{1,n}$ because no vertex of $G_{A^{(n-1)}}$ is linked to $v_n$. However, the erstwhile vertex $v_{(n-1),1}$ of $G_{A^{(n-1)}}$ is also linked to all the \textit{replica} vertices hence, inducing the complete graph $K_{2^{n-1}}$. Clearly, another largest complete graph does not exist.
\end{proof}

The primitive hole number of a set-graph is determined recursively in the following theorem.

\begin{thm}
	For a set-graph $G_{A^{(n)}}, n \geq 3$ we find the number of primitive holes through the recursive formula $h(G_{A^{(n+1)}}) = h(G_{A^{(n)}}) + \binom{2^n}{3} + 4|E(G_{A^{(n)}})|$.
\end{thm}
\begin{proof}
	Consider the set-graph $G_{A^{(n)}}$ whose primitive hole number is denoted by $h(G_{A^{(n)}})$. Extending $G_{A^{(n)}}$ to $G_{A^{(n+1)}}$ will only increase the number of primitive holes. What we need here to determine the number of additional primitive holes formed on extending $G_{A^{(n)}}$ to $G_{A^{(n+1)}}$. This calculation is done as follows.
	
	\vspace{0.2cm}
	
	The set of replica vertices together with the vertex $v_{(n+1),1}$ induce a complete subgraph $K_{|V(G_{A^{(n)}}| +1}$ and hence an additional $\binom{2^n}{3}$ primitive holes are added to the extended graph. Finally, for each edge $v_{s,i}v_{t,j} \in E(G_{A^{(n)}})$ the vertices $v_{s,i}, v_{t,j}, v^{\ast}_{s,i}, v^{\ast}_{t,j}$ induce a $K_4$ subgraph and the number of primitive holes thus formed is $\binom{4}{3} = 4$. Hence, a further $4|E(G_{A^{(n)}}|$ primitive holes are added to the extended graph. This completes the proof.
\end{proof}

\ni Next, we introduce the following notions for a set-graph as follows.

\begin{defn}{\rm
		Let $G=G_{A^{(n)}}$ be a set-graph on a non-empty set $A^{(n)}$ and let $A_{s,i}^{(n)}$ be an arbitrary subset of the set $A^{(n)}$. The \textit{characteristic function} of a subset $A_{t,j}^{(n)}$ of $A^{(n)}$ with respect to $A_{s,i}^{(n)}$, denoted by $\xi_{A_{s,i}^{(n)}}(A_{t,j}^{(n)})$, is defined as 
		\begin{equation*}
		\xi_{A_{s,i}^{(n)}}(A_{t,j}^{(n)})=
		\begin{cases}
		1 & \text{if} \quad A_{s,i}^{(n)}\cap A_{t,j}^{(n)}\ne \emptyset\\
		0 & \text{if} \quad A_{s,i}^{(n)}\cap A_{t,j}^{(n)}=\emptyset.
		\end{cases}
		\end{equation*} }
\end{defn}

\begin{defn}{\rm
		The \textit{tightness number} of a subset $A^{(n)}_{s,k}$, denoted $\varsigma(A_{s,k}^{(n)})$ is the number of subsets distinct from $A_{s,k}^{(n)}$ for which the intersection with $A_{s,k}^{(n)}$ is non-empty. Hence, $\varsigma(A_{s,k}^{(n)}) = \sum\limits_{j\ne k}\xi_{A_{s,i}^{(n)}}(A_{t,j}^{(n)})$.}
\end{defn}

We note that in terms of the definition of a set-graph and for the vertex $v_{s,i}$ corresponding to the subset $A_{s,k}^{(n)}$ we have, $d_{G_{A^{(n)}}}(v_{s,i}) = \varsigma(A_{s,k}^{(n)})$. Also, we have that $|E(G_{A^{(n)}})| = \frac{1}{2}\sum\limits_{1 \le k \le 2^n-1}\varsigma(A_{s,k}^{(n)})$.

\vspace{0.25cm}

The next theorem enables us to employ a step-wise recursive formula to determine the tightness number of all non-empty subsets of $A^{(n+1)}$ if the the tightness number of all non-empty subsets of $A^{(n)}$ are known.

\begin{thm}
	Consider a set-graph $G_{A^{(n)}}, n \ge 1$ and its extended set-graph $G_{A^{(n+1)}}$. Then we have 
	\begin{enumerate}\itemsep0mm
		\item[(i)] $\varsigma (\{a_{n+1}\}) = 2^n - 1$,
		\item[(ii)] For each erstwhile subset $A^{(n)}_{s,i}$ with $\varsigma (A^{(n)}_{s,i}) = k$ in $G_{A^{(n)}}$, we have $\varsigma (A_{s,i}^{(n+1)}) = 2k+1$ in $G_{A^{(n+1)}}$,
		\item[(iii)] For a replica vertex say, $v^{\ast}_{s,i}$ representing the new subset $A_{s,i}^{(n)} \cup \{a_{n+1}\}$ we have $\varsigma (A_{s,i}^{(n)}\cup \{a_{n+1}\}) = 2^n+k$.
	\end{enumerate}
\end{thm}
\begin{proof} Let $G_{A^{(n)}}$ be a set-graph and $G_{A^{(n+1)}}$ be its extended graph obtained by introducing a new element, say $a_{n+1}$, to the set $A^{(n)}$. Then, 
	
	\begin{enumerate}\itemsep0.25cm
		\item[(i)] To generate the set-graph $G_{A^{(n+1)}}$ by extending the set-graph $G_{A^{(n)}}$, we initially add the subsets $\{a_{n+1}\}$ and $A_{s,i}^{(n)} \cup \{a_{n+1}\}$, for all applicable values of $s$ and $i$. Clearly, $\{a_{n+1}\} \cap (A_{s,i}^{(n)}\cup \{a_{n+1}\}) \ne \emptyset$. So, $\varsigma (\{a_{n+1}\}) \ge 2^n -1$. Also,  $\{a_{n+1}\} \cap A_{s,i}^{(n)}=\emptyset$. Therefore, we have $\varsigma (\{a_{n+1}\}) = 2^n -1$.
		
		\item[(ii)] If $\varsigma (A_{s,i}^{(n)}) = k$ in $G_{A^{(n)}}$, the subset $A_{s,i}^{(n)}$ has non-zero intersections with exactly $k$ distinct subsets of $A^{(n)}$. Since in the replication, we have $A_{s,i}^{(n)}\cup \{a_{n+1}\}$ together with the subsets $A_k^{(n)}\cup \{a_{n+1}\}$, for all $k$,  the result follows.
		
		\item[(iii)] The replica vertices are $2^n-1$ in number and induce a complete graph together with vertex $v_{1,(n+1)}$. This partially represents $2^n -1$ non-zero intersections in respect of any subset say, $A_{s,i}^{(n)}\cup \{a_{n+1}\}$ corresponding to any replica vertex say, $v^{\ast}_{s,i}$. Clearly, $(A_{s,i}^{(n)} \cup \{a_{n+1}\}) \cap A_{s,i}^{(n)} \ne \emptyset$ and $(A_{s,i}^{(n)}\cup \{a_{n+1}\}) \cap A_{t,j}^{(n)}\ne \emptyset$, for all $A_{s,i}^{(n)}\cap A_{t,j}^{(n)}\ne \emptyset$ . It implies that an additional $(k+1)$ non-zero intersections exist in respect of $A_{s,i}^{(n)}\cup \{a_{n+1}\}$. Hence, $\varsigma (A_{s,i}^{(n)}\cup \{a_{n+1}\}) = 2^n-1+(k+1)= 2^n+k$.
	\end{enumerate}
	\ni This completes the proof.
\end{proof}

\section{Certain Parameters of Set-Graphs}

Let $G$ be a given non-trivial finite graph. The \textit{chromatic number}, denoted by $\chi(G)$, of $G$ is the minimum $k$ for which $G$ is $k$-colourable.

\begin{thm}
	The chromatic number of a set-graph $G_{A^{(n)}}$ is $\chi(G_{A^{(n)}}) = 2^{n-1}$.
\end{thm}
\begin{proof}
	It is easy to see that $\chi(G_{A^{(1)}}) = 1 = 2^{1-1}$, $\chi(G_{A^{(2)}})=2=2^{2-1}$, $\chi(G_{A^{(3)}})=4= 2^{3-1}$ (See figure \ref{fig-1}). Assume the result holds for the set-graph $G_{A^{(k)}}$. Therefore, we have $\chi(G_{A^{(k)}}) = 2^{k-1}$.
	
	\vspace{0.2cm}
	
	Now consider the set-graph $G_{A^{(k+1)}}$. From the steps to be followed to extend from $G_{A^{(k)}}$ to $G_{A^{(k+1)}}$ (See proof of Theorem \ref{T-SGRD}), we have the erstwhile vertices of $G_{A^{(k)}}$, and in addition, the replica vertices corresponding to the vertices of $G_{A^{(k)}}$ and one more vertex $v_{1,(k+1)}$. From the proof of Proposition \ref{P-SGKN} we can notice that the replica vertices and vertex $v_{1,(k+1)}$ induce a largest complete subgraph, $K_{2^k}$ in the extended graph of $G_{A^{(n)}}$. We also note that the replica vertices and vertex $v_{1,k}$ form a second largest complete subgraph, $K_{2^k}$. 
	
	\vspace{0.2cm}
	
	Since the vertices $v_{1,k}$ and $v_{1,{k+1}}$ are not adjacent in $G_{A^{(n)}}$, both of them have the same colour, say $c_1$ and colour the replica vertices by the colours $c_2, c_3, c_4, \ldots, c_{2^k}$. Since, no other largest complete graph exists it is always possible to find at least one pair of erstwhile-replica vertices which are non-adjacent. Hence the erstwhile vertex can carry the colour of such a replica vertex. This can be done in such a way that two adjacent erstwhile vertices do not carry the same colour by using the colours $c_2, c_3, c_4, ..., c_{2^k}$ accept for the colour of $v_{n,1}$, exhaustively. So the result $\chi(G_{A^{(k+1)}}) = 2^k = 2^{(k+1)-1}$ follows. Hence, the main result follows by induction.
\end{proof}

An \textit{independent set} of  graph $G$ is a set of mutually non-adjacent vertices of $G$. The \textit{independence number}, denoted by $\alpha(G)$, of $G$ is the cardinality of a maximal independent set of $G$. The independence number of a set-graph is determined in the following theorem.

\begin{thm}
	The independence number of a set-graph $G_{A^{(n)}}$ is $\alpha(G_{A^{(n)}}) = n$.
\end{thm}
\begin{proof}
	Let $G= G_{A^{(n)}}$ be a given set-graph. Then, as explained in the proof of Theorem \ref{T-SGDV}, the vertices $v_{1,1}, v_{1,2}, v_{1,3}, ..., v_{1,n}$, corresponding to the singleton subsets of $A^{(n)}$, are pairwise non-adjacent. Hence, the set $I= \{v_{1,1}, v_{1,2}, v_{1,3}, ..., v_{1,n}\}$ is an independent set. By the Definition \ref{D-SG}, we note that any vertex in $V(G)-I$ is adjacent to at least one vertex in $I$. Therefore, $I$ is the maximal set of mutually non-adjacent vertices and hence is the maximal independent set in $G$. Hence, $\alpha(G_{A^{(n)}}) = n$.
\end{proof}

A \textit{domianting set} of a graph is a set of vertices $D$ such that every vertex of $G$ is either in $D$ or is adjacent to at least one vertex in $D$. The \textit{domination number}, denoted by $\gamma(G)$, of a graph $G$ is the cardinality of the minimal dominating set of $G$. The following discusses the domination number of a set-graph.

\begin{thm}
	The domination number of a set-graph$G_{A^{(n)}}$ is $\gamma(G_{A^{(n)}}) = 1$.
\end{thm}
\begin{proof}
	Let $G=G_{A^{(n)}}$ be a given set-graph. By Corollary \ref{C-UDG}, the vertex $v_{n,1}$, corresponding to the $n$-element set $A^{(n)}$, is the unique vertex in the set-graph $G$ that is adjacent to all other vertices of the set-graph $G$.  Therefore, the singleton set $\{v_{n,1}\}$ is the minimal set such that every vertex of $G$ is adjacent to the unique element in the set $D$. Therefore, $\gamma(G_{A^{(n)}})=1$.
\end{proof}

Another parameter we consider here is the bondage number of a graph $G$, which is denoted by $b(G)$ defined as the minimum number of edges to be removed to increase the domination number $\gamma(G)$ by $1$.

\begin{thm}
	The bondage number of a set-graph $G_{A^{(n)}}$ is $b(G_{A^{(n)}}) =1$.
\end{thm}
\begin{proof}
	Since $\{v_{n,1}\}$ is the minimal dominating set of the set-graph $G=G_{A^{(n)}}$, the removal of any edge  $v_{s,i}v_{n,1}$ will increase the domination number by $1$ in the reduced graph $G_{A^{(n)}} - v_{s,i}v_{n,1}$.
\end{proof}

Another parameter we are going to discuss here is the McPhersion Number of undirected graphs. For this, let us now recall the definition McPhersion Number, as given in \cite{JS1}. 

\begin{defn}{\rm 
		\cite{JS1} The \textit{McPherson recursion} is a series of \textit{vertex explosions} such that on the first iteration a vertex $v \in V(G)$ explodes to arc (directed edges) to all vertices $u \in V(G)$ for which the edge $vu \notin E(G)$, to obtain the mixed graph $G'_1$. Now $G'_1$ is considered on the second iteration and a vertex $w \in V(G'_1) = V(G)$ may explode to arc to all vertices $z \in V(G'_1)$ if edge $wz \notin E(G)$ and arc $(w,z)$ or $(z,w) \notin E(G'_1)$.
		
		\vspace{0.25cm}
		
		The \textit{McPherson number}, denoted by $\Upsilon(G)$,  of a simple connected graph $G$ is the minimum number of iterative vertex explosions say $l$, to obtain the mixed graph $G'_l$ such that the underlying graph $G^{\ast}_l \cong K_n$.}
\end{defn}

The McPherson number of a set-graph $G_{A^{(n)}}$ is determined in the following theorem. 

\begin{thm}
	For a set-graph $G_{A^{(n)}}$ we have $\Upsilon(G_{A^{(n)}}) = 2^{n-1}-1$.
\end{thm}
\begin{proof}
	Consider the set-graph $G_{A^{(n-1)}}$ which has $2^{n-1}-1$ vertices. On extending the set-graph $G_{A^{(n)}}$, the replica vertices, as explained in Theorem \ref{T-SGRD}, together with vertex $v_{n,1}$ induce a complete subgraph and hence no further vertex explosions are required to ensure complete induced by these vertices. However, at least all the erstwhile vertices require vertex explosions to ensure complete connectivity amongst themselves and the replica vertices. Hence, $\Upsilon(G_{A^{(n)}}) = 2^{n-1} -1$. 
\end{proof}

\section{Conclusion}
We have discussed particular types of graphs called set-graphs and studied certain characteristics and structural properties of these graphs. The study seems to be promising as it can be extended to certain standard graph classes and certain graphs that are associated with the given graphs. More problems in this area are still open and hence there is a wide scope for further studies. Now, we have the notion of {M\`ela numbers} as follows.  

\begin{defn}{\rm
		The set of \textit{M\`ela numbers} is defined to be the set $\M =\{m_i: m_1 = 1, m_i = 2m_{i-1} + 1, i\in \N, i\ge 2\}$.}
\end{defn}

\ni Invoking the above definition, it can immediately be noted that $|V(G_{A^n})| = m_n$.

\vspace{0.2cm}

\ni Some open problems\footnote{The first author wishes to dedicate these open problems to Ms. M\`ela Odendaal, who is expected to grow up as a great mathematician.} we wish to mention in this context are the following.

\begin{prob}
	Show that $m_i + m_j \notin \M,$ $m_im_j \notin \M$ and if $m_i > m_j$ then $m_i - m_j \notin \M$. 
\end{prob} 

\begin{prob}
	Show that $m_{ki}$ is divisible by $m_i$ but, $\frac{m_{ki}}{m_i} \notin \M$. 
\end{prob}

Finding other number theoretical results for \textit{M\`ela numbers} are also challenging problems which seems to be promising. All these facts indicate that there is a wide scope for further research in this area.




\begin{thebibliography}{25}
	
	\bibitem{BM} J. A. Bondy and U. S. R. Murty, {\bf Graph Theory with Applications}, Macmillan Press, London, 1976.
	
	\bibitem{CL1} G. Chartrand and L. Lesniak, {\bf Graphs and Digraphs}, CRC Press, 2000.
	
	\bibitem{GY} J. T. Gross and J. Yellen, {\bf Graph Theory and its Applications}, CRC Press, 2006. 
	
	\bibitem{GL} S.S. Gupta and T. Liang, {\em On a Sequential Subset Selection Procedure}, Technical report \#88-23, Department of Statistics, Purdue University, Indiana, U.S., 1988.
	
	\bibitem{FH}  F. Harary, {\bf Graph Theory}, Addison-Wesley, 1994.
	
	\bibitem{KS1} J. Kok and N. K. Sudev, {\em A Study on Primitive Holes of Certain Graphs}, International Journal of Scientific and Engineering Research, {\bf 6}(3)(2015), 631-635.
	
	\bibitem{JS1} J. Kok and Susanth C., {\em Introduction to the McPherson Number $\Upsilon(G)$ of a Simple Connected Graph}, Pioneer Journal of Mathematics and Mathematical Sciences, {\bf 13}(2), (2015), 91-102.
	
	\bibitem{MM} T. A. McKee and F. R. McMorris, {\bf Topics in Intersection Graph Theory}, SIAM, Philadelphia, 1999.
	
	\bibitem{KHR} K. H. Rosen, {\bf Handbook of Discrete and Combinatorial Mathematics}, CRC Press, 2000. 
	
	\bibitem{DBW} D. B. West, {\bf Introduction to Graph Theory}, Pearson Education Inc., 2001.
	
\end{thebibliography}
\end{document}